\documentclass{article}
\usepackage{amsmath, amsthm}
\usepackage{amssymb}
\newtheoremstyle{theorem}
  {10pt}          
  {10pt}  
  {\sl}  
  {\parindent}     
  {\bf}  
  {. }    
  { }    
  {}     
\theoremstyle{theorem}
\newtheorem{theorem}{Theorem}

\newtheorem{lemma}[theorem]{Lemma}
\newtheorem{remark}[theorem]{Remark}
\newtheorem{proposition}[theorem]{Proposition}
\newtheorem{example}[theorem]{Example}
\newtheoremstyle{defi}
  {10pt}          
  {10pt}  
  {\rm}  
  {\parindent}     
  {\bf}  
  {. }    
  { }    
  {}     
\theoremstyle{defi}
\newtheorem{definition}[theorem]{Definition}

\begin{document}

\title{ Caputo  q-Fractional Initial Value Problems and a q-Analogue
Mittag-Leffler Function }

\author{Thabet Abdeljawad and Dumitru Baleanu\footnote{On leave of absence
from Institute of Space Sciences, P.O.BOX, MAG-23, R
76900,Maturely-Bucharest, Romania, Emails: dumitru@cankaya.edu.tr,
baleanu@venus.nipne.ro}\\
 Department of Mathematics and
Computer Science\\
  \c{C}ankaya University, 06530 Ankara, Turkey}
\maketitle
\begin{abstract}
Caputo q-fractional derivatives are introduced and studied. A Caputo -type
q-fractional initial value problem is solved and its solution is expressed
by means of a new introduced q-Mittag-Leffler function. Some open problems
about q-fractional integrals are  proposed as well.



\end{abstract}

\section{Introduction } \label{s:1}

The concept of fractional calculus is not new. It is believed to have
stemmed from a question raised in 1695. However, it has gained
considerable popularity and importance during the last three decades or
so. This is due to its distinguished applications in numerous diverse
fields of science and engineering (\cite{Samko}, \cite{Podlubny},
\cite{Kilbas}). The q-calculus is also not of recent appearance. It was
initiated in twenties of the last century. As a survey about this calculus
we refer to \cite{history}. Starting from the q-analogue of Cauchy formula
\cite{Al-salam2}, Al-Salam started the fitting of the concept of
q-fractional calculus. After that he (\cite{Alsalaml}, \cite{Alsalam}) and
Agarwal R. \cite{Agarwal} continued on by studying certain q-fractional
integrals and derivatives, where they proved the semigroup properties for
left and right (Riemann)type fractional integrals but without variable
lower limit and variable upper limit, respectively. Recently, the authors
in \cite{Pred} generalized the notion of the (left)fractional q-integral
and q-derivative by introducing variable lower limit and proved the
semigroup properties. However, the case of the (right) q-fractional
integral by introducing a variable upper limit is still open. This open
problem will be stated clearly in this article.

Very recently and after the appearance of time scale calculus (see for
example \cite{Boh}), some authors started to pay attention and apply the
techniques of time scale to discrete fractional calculus
(\cite{Ferd},\cite{Feri},\cite{Ferq}, \cite{Th}) benefitting from the
results announced before in \cite{Miller}. All of these results are mainly
about fractional calculus on the time scales $T_q=\{q^n:n \in
\mathbb{Z}\}\cup \{0\}$ and $h\mathbb{Z}$ \cite{Bastos}. However, the
study of fractional calculus on time scales combining the previously
mentioned time scales is still unknown. Continuing in this direction and
being motivated by all above, in this article we  define and study Caputo
type q-fractional derivatives. This manuscript is organized as follows:
Section 2 contains essential definitions and results  about fractional
q-integrals and q-derivatives, where we present an open problem about the
semigroup property. Section 3 is devoted to define and study left and
right Caputo q-fractional derivatives. In Section 4, we solve a Caputo
q-fractional nonhomogenuous linear dynamic equation, where the solution is
expressed by a q-analogue of Mittag-Leffler function.

\section {Preliminaries and Essential Results about q-Calculus, Fractional
q-Integrals and q-Derivatives}\label{s:3}

 For the theory of q-calculus we refer the reader to the survey
\cite{history} and for the basic definitions and results for the
q-fractional calculus we refer to \cite{Ferq}. Here we shall
summarize some of those basics.

For $0<q<1$, let $T_q$ be the time scale
$$T_q=\{q^n:n \in \mathbb{Z}\}\cup \{0\}.$$ where $Z$ is the set of
integers. More generally, if $\alpha$ is a nonnegative real number
then we define the time scale
$$T_q^\alpha=\{q^{n+\alpha}:n \in Z\}\cup \{0\},$$ we write $T_q^0=T_q.$

 For a function $f:T_q\rightarrow \mathbb{R}$, the nabla
q-derivative of $f$ is given by
\begin{equation} \label{qd}
\nabla_q f(t)=\frac{f(t)-f(qt)}{(1-q)t},~~t \in T_q-\{0\}
\end{equation}
The nabla q-integral of $f$ is given by
\begin{equation} \label{qi}
\int_0^t f(s)\nabla_q s=(1-q)t\sum_{i=0}^\infty q^if(tq^i)
\end{equation}
and for $0\leq a \in T_q$

$$\int_a^t f(s)\nabla_q s=\int_0^t f(s)\nabla_q s - \int_0^a f(s)\nabla_q s$$
On the other hand
\begin{equation} \label{r1}
\int_t^\infty f(s)\nabla_q s=(1-q)t\sum_{i=1}^\infty q^{-i}
f(tq^{-i})
\end{equation}
 and for $0<b<\infty$ in $T_q$
\begin{equation} \label{r2}
\int_t^b f(s)\nabla_q s=\int_t^\infty f(s)\nabla_q s - \int_b^\infty
f(s)\nabla_q s
\end{equation}
 By the fundamental theorem in q-calculus we have
\begin{equation} \label{fq}
\nabla_q \int_0^t f(s)\nabla_q s=f(t)
\end{equation}
and if $f$ is continuous at $0$, then
\begin{equation} \label{cfq}
\int_0^t \nabla_qf(s)\nabla_q s=f(t)-f(0)
\end{equation}
Also the following identity will be helpful
\begin{equation} \label{help}
\nabla_q\int_a^t f(t,s)\nabla_q s=\int_a^t \nabla_qf(t,s)\nabla_q
s+f(qt,t)
\end{equation}
Similarly the following identity will be useful as well
\begin{equation} \label{help1}
\nabla_q\int_t^b f(t,s)\nabla_q s=\int_{qt}^b \nabla_qf(t,s)\nabla_q
s-f(t,t)
\end{equation}
The q-derivative in (\ref{help}) and (\ref{help1}) is applied with
respect to t.

From the theory of q-calculus and the theory of time scale more
generally, the following product rule is valid

\begin{equation} \label{qproduct}
\nabla_q (f(t)g(t))=f(qt)\nabla_q g(t)+\nabla_q f(t)g(t)
\end{equation}
The q-factorial function for $n\in \mathbb{N}$ is defined by
\begin{equation} \label{qfact}
(t-s)_q^n=\prod_{i=0}^{n-1}(t-q^is)
\end{equation}
When $\alpha$ is a non positive integer, the  q-factorial function
is defined by
\begin{equation} \label{qfactg}
(t-s)_q^\alpha=t^\alpha\prod_{i=0}^\infty \frac{1- \frac{s}{t} q^i}
{1- \frac{s}{t} q^{i+\alpha}}
\end{equation}
We summarize some of the properties of  q-factorial functions, which
can be found mainly in \cite{Ferq}, in the following lemma

\begin{lemma} \label{qproperties}
(i)$ (t-s)_q^{\beta+\gamma}=(t-s)_q^\beta (t-q^\beta s)_q^\gamma$

(ii)$(at-as)_q^\beta=a^\beta (t-s)_q^\beta$

(iii) The nabla q-derivative of the q-factorial function with
respect to $t$ is

$$\nabla_q (t-s)_q^\alpha
=\frac{1-q^\alpha}{1-q}(t-s)_q^{\alpha-1}$$

(iv)The nabla q-derivative of the q-factorial function with respect
to $s$ is
$$\nabla_q (t-s)_q^\alpha
=-\frac{1-q^\alpha}{1-q}(t-qs)_q^{\alpha-1}$$ where
$\alpha,\gamma,\beta \in \mathbb{R}.$
\end{lemma}
For the q-gamma function, $\Gamma_q(\alpha)$, we refer the reader to
\cite{Ferq} and the references therein. We just mention here  the
identity

\begin{equation} \label{qid}
\Gamma_q(\alpha+1)=\frac{1-q^\alpha}{1-q}\Gamma_q(\alpha),~~\Gamma_q(1)=1,~\alpha
>0.
\end{equation}
The authors in \cite{Ferq} following \cite{Agarwal} defines the left
fractional q-integral of order $\alpha\neq 0,-1,-2,...$ by
\begin{equation} \label{ag}
 _{q}I^\alpha f(t)=\frac{1}{\Gamma_q(\alpha)}\int_0^t
 (t-qs)_q^{\alpha-1}f(s)\nabla_qs
\end{equation}

In \cite{Agarwal} it was proved that the left q-fractional integral
obeys the identity
\begin{equation}\label{t1}
 _{q}I^\beta ~ _{q}I^\alpha f(t)=
 _{q}I^{\alpha+\beta}f(t),~~~~\alpha~,\beta>0
\end{equation}
The left q-fractional integral $ _{q}I_a^\alpha$ starting from
$0<a\in T_q$ is to be defined by

\begin{equation} \label{t3}
 _{q}I_a^\alpha f(t)=\frac{1}{\Gamma_q(\alpha)}\int_a^t
 (t-qs)_q^{\alpha-1}f(s)\nabla_qs
\end{equation}

It is clear, from the q-analogue of Cauchy's formula
\cite{Al-salam2}, that

\begin{equation}\label{t11}
 \nabla_q^n ~_{q}I_a^n f(t)=f(t)
\end{equation}
  where $n$ is
a positive integer and $0\leq a \in T_q$

Recently, in Theorem 5 of \cite{Pred}, the authors there have proved
that
\begin{equation}\label{t4}
 _{q}I_a^\beta ~ _{q}I_a^\alpha f(t)=
 _{q}I_a^{\alpha+\beta}f(t),~~~~\alpha~,\beta>0
\end{equation}

The right q-fractional integral of order $\alpha$ is defined by
\cite{Agarwal}

\begin{equation}\label{t55}
I_q^ \alpha
f(t)=\frac{q^{-(1/2)\alpha(\alpha-1)}}{\Gamma_q(\alpha)}\int_t^\infty
(s-t)_q^{\alpha-1}f(sq^{1-\alpha})\nabla_qs
\end{equation}
and the right q-fractional integral of order $\alpha$
ending at $b$ for some $b \in T_q$ is defined by
\begin{equation}\label{t555}
_{b}I_q^ \alpha
f(t)=\frac{q^{-(1/2)\alpha(\alpha-1)}}{\Gamma_q(\alpha)}\int_t^b
(s-t)_q^{\alpha-1}f(sq^{1-\alpha})\nabla_qs
\end{equation}
Note that, while the left q-fractional integral $_{q}I_a^\alpha$
maps functions defined $T_q$ to functions defined on $T_q$, the
right q-fractional integral $_{b}I_q^\alpha$, $0<b\leq\infty$,  maps
functions defined on $T_q^{1-\alpha}$ to functions defined on $T_q$.

 It is clear, from the q-analogue of Cauchy's formula
\cite{Al-salam2}, that

\begin{equation}\label{s1}
 \nabla_q^n I_q^n f(t)=(-1)^nf(t)
\end{equation}

In \cite{Alsalam} it was proved that the right q-fractional integral
obeys the identity
\begin{equation}\label{s2}
 I_q^\beta ~ I_q^\alpha f(t)=
 I_q^{\alpha+\beta}f(t),~~~~\alpha~,\beta>0
\end{equation}
Taking into account the domain and the range of the right
q-fractional integral, as mentioned above, we note that the formula
(\ref{s2}) is valid under the condition that $f$ must be at least defined
on  $T_q$,
$T_q^{1-\beta}$,   $T_q^{1-\alpha}$
 and  $T_q^{1-(\alpha+\beta)}$.

 A particular case of the identity (\ref{s2}) is
\begin{equation}\label{s22}
 I_q^{n-\alpha} ~ I_q^\alpha f(t)=
 I_q^nf(t),~~~~\alpha~>0.
\end{equation}

\begin{lemma} \label{remb1}
For $\alpha,\beta > 0$ and a function $f$ fitting suitable domains,
we have
\begin{equation}
\int_b^\infty (t-x)_q^{\beta-1}~_{b}I_q^\alpha
f(tq^{1-\beta})\nabla_qt=0
\end{equation}
\end{lemma}
\begin{proof}
From (\ref{r1}) we can write
$$\int_b^\infty (t-x)_q^{\beta-1}~_{b}I_q^\alpha
f(tq^{1-\beta})\nabla_qt=$$
\begin{equation} \label{eq1}
(1-q)b \sum_{i=0}^\infty
q^{-i}(bq^{-i}-x)_q^{\beta-1}~_{b}I_q^\alpha f(q^{1-\beta}bq^{-i})
\end{equation}
From the fact that $(t-r)_q^{\beta-1}=0$, when $t<r$ we conclude
that $_{b}I_q^\alpha f(q^{1-\beta}bq^{-i})=0$ and hence the result
follows.
\end{proof}

\textbf{Problem 1}: Can we use Lemma \ref{remb1}  and following
 similar ideas to  that in \cite{Pred} to prove that

\begin{equation}\label{s222}
 _{b}I_q^\beta ~ _{b}I_q^\alpha f(t)=~
 _{b}I_q^{\alpha+\beta}f(t),~~~~\alpha~,\beta>0, ~0<b\in T_q
\end{equation}
Alternatively, can we define the q-analogue of the Q-operator and prove
that $Q ~_{q}I_a^\alpha f(t) =~ _{b}I_q^\alpha Q f(t)$? Then apply the
Q-operator to the identity

\begin{equation}\label{t44}
 _{q}I_a^\beta ~ _{q}I_a^\alpha g(t)=
 _{q}I_a^{\alpha+\beta}g(t),~~~~\alpha~,\beta>0
\end{equation}
with $g(t)=Qf(t)$ to obtain (\ref{s222}). Recall that in the continuous
case $Qf(t)=f(a+b-t)$.

In connection to Problem 1, the following open problem is also raised

\textbf{Problem 2}: Is it possible to obtain a by-part formula for
q-fractional derivatives when the lower limit $a$ and the upper limit $b$
both exist. That is on the interval $[a,b]_q$. As for the $(0,\infty)$
case there is  a formula was early obtained by Agarwal in \cite{Agarwal}.

As for the left and right (Riemann) q-fractional derivatives of
order $\alpha > 0$,as traditionally  done in fractional calculus,
they are defined respectively by

\begin{equation} \label{lrqd}
_{q}\nabla_a^\alpha f(t)\triangleq
\nabla_q^n~_{q}I_a^{n-\alpha}f(t)~\texttt{and}~_{b}\nabla_q^\alpha
f(t)\triangleq (-1)^n \nabla_q^n~_{b}I_q^{n-\alpha}f(t)
\end{equation}
where $n=[\alpha]+1$ and $a,b\in T_q\cup \{\infty\}$ with $0\leq a <
b\leq \infty$. We usually remove the endpoints in the notation when
$a=0$ or $b=\infty$. Here, we point that the operator
$_{q}\nabla_a^\alpha$ maps functions defined on $T_q$ to functions
defined on $T_q$, while the operator $ _{b}\nabla_q^\alpha$ maps
functions defined on $T_q^{1-(n-\alpha)}$ to functions defined on
$T_q$. Also, particularly, one has to note that
\begin{equation} \label{lrqd}
_{q}\nabla_a^n f(t)=\nabla_q^n f(t)~\texttt{and}~_{b}\nabla_q^n
f(t)=(-1)^n \nabla_q^n f(t)
\end{equation}
where $\nabla_q^n$ always denotes the $n-th$ q-derivative (i.e. the
q-derivative applied n times).

\section{Caputo q-Fractional Derivative } \label{s:3}
In this section, before defining Caputo-type q-fractional
derivatives and relating them to Riemann ones, we first state and
prove some essential preparatory lemmas.
\begin{lemma} \label{m}
For any $\alpha>0$, the following equality holds:
\begin{equation} \label{m1}
_{q}I_a ^\alpha\nabla_q f(t)= \nabla_q ~ _{q}I_a ^\alpha
f(t)-\frac{(t-a)_q^{\alpha-1}}{\Gamma_q(\alpha)}f(a)
\end{equation}

\end{lemma}
\begin{proof}
From (\ref{qproduct}) and (iv) of Lemma  \ref{qproperties}, we
obtain the following q-integration by parts:
\begin{equation} \label{mm1}
\nabla_q ((t-s)_q^{\alpha-1}f(s))= (t-qs)_q^{\alpha-1}\nabla_q f(s)-
\frac{1-q^{\alpha-1}}{1-q}(t-qs)_q^{\alpha-2}f(s)
\end{equation}
Applying (\ref{mm1}) leads to
\begin{equation} \label{m2}
_{q}I_a ^\alpha\nabla_q f(t)=\frac{(t-s)_q
^{\alpha-1}}{\Gamma_q(\alpha)}f(s)|_a^t
+\frac{1-q^{\alpha-1}}{1-q}\int_a^t (t-qs)_q^{\alpha-2}f(s)\nabla_qs
\end{equation}
or
$$
_{q}I_a ^\alpha\nabla_q f(t)= - \frac{(t-a)_q ^{\alpha-1} }{\Gamma_q
(\alpha)}f(a)+ \frac{1-q^{\alpha-1}}{1-q}\int_a^t
(t-qs)_q^{\alpha-2}f(s)\nabla_qs
$$
On the other hand, and by the help of (iii) of Lemma
\ref{qproperties}, (\ref{help}) and the identity (\ref{qid}), we
find that
\begin{equation} \label{m4}
\nabla_q ~ _{q}I_a ^\alpha f(t)=\frac{1-q^{\alpha-1}}{1-q}\int_a^t
(t-qs)_q^{\alpha-2}f(s)\nabla_qs,
\end{equation}
which completes the proof.
\end{proof}

\begin{theorem} \label{qinter}
For any real $\alpha>0$ and any positive integer $p$ such that
$\alpha -p+1$ is not negative integer or 0, in particular $\alpha >p-1$,
the following equality holds:
\begin{equation} \label{qinter1}
_{q}I_a ^\alpha \nabla_q^p f(t)= \nabla_q^p~ _{q}I_a ^\alpha
f(t)-\sum_{k=0}^{p-1}\frac{(t-a)_q^{\alpha-p+k}}{\Gamma_q(\alpha+k-p+1)}\nabla_q^k
f(a)
\end{equation}
\end{theorem}
\begin{proof}
The proof can be achieved by following inductively on $p~$   and
making use of Lemma \ref{m}, (iii) of Lemma \ref{qproperties} and
(\ref{qid}).
\end{proof}
Now we obtain an analogue to Lemma  \ref{m} for the right
q-integrals.
\begin{lemma} \label{rm}
For any $\alpha > 0$, the following equality holds:

\begin{equation}\label{th}
_{q^{-1}b}I_q^\alpha~_{b}\nabla_q f(t)=~_{b}\nabla_q~_{b}I_q^\alpha
f(t)-\frac{r(\alpha)}{\Gamma_q(\alpha)}(b-qt)_q^{\alpha-1}f(q^{1-\alpha}q^{-1}b)
\end{equation}
where
\begin{equation}
r(\alpha)= q^{(-1/2)\alpha(\alpha-1)}
\end{equation}
 and
 $$_{b}\nabla_q f(t)=-\nabla_q f(t)$$.
\end{lemma}
\begin{proof}
First, by the help of (iii) of Lemma \ref{qproperties} and
(\ref{qproduct}), the following q-calculus by-parts version is
valid: $$(s-t)_q^{\alpha-1}\nabla_qf(sq^{1-\alpha})q^{1-\alpha}=$$

\begin{equation} \label{rid}
\nabla_q((s-t)_q^{\alpha-1}f(sq^{1-\alpha}))-\frac{1-q^{\alpha-1}}{1-q}(s-t)_q^{\alpha-2}
f(sq^{2-\alpha})
\end{equation}
where the q-derivative is applied with respect to $s$. Using
(\ref{rid}) we obtain
 $$ _{q^{-1}b}I_q^\alpha~_{b}\nabla_q f(t)=$$
\begin{equation} \label{m1}
\frac{q^{\alpha-1}r(\alpha)}{\Gamma_q(\alpha)}(\frac{1-q^{\alpha-1}}{1-q}\int_t^{q^{-1}b}
(s-t)_q^{\alpha-2}f(q^{2-\alpha}s)\nabla_qs-(s-t)_q^{\alpha-1}f(q^{1-\alpha}s)|_t^{q^{-1}b})
\end{equation}
\begin{equation} \label{m2}
=\frac{ q^{\alpha-1}r(\alpha) }
{\Gamma_q(\alpha)}(\frac{1-q^{\alpha-1}}{1-q}\int_t^{q^{-1}b}
(s-t)_q^{\alpha-2}f(q^{2-\alpha}s)\nabla_qs-(q^{-1}b-t)_q^{\alpha-1}f(q^{1-\alpha}q^{-1}b))
\end{equation}
On the other hand (\ref{help1}) and (iv) of Lemma \ref{qproperties}
imply

\begin{equation} \label{m3}
~_{b}\nabla_q~ _{b}I_q^\alpha f(t)=\frac{ q^{\alpha-1}r(\alpha) }
{\Gamma_q(\alpha)}\frac{1-q^{\alpha-1}}{1-q}\int_t^{q^{-1}b}
(s-t)_q^{\alpha-2}f(q^{2-\alpha}s)\nabla_qs
\end{equation}
Taking into account (\ref{m2}) and (\ref{m3}), identity (\ref{th})
will follow and the proof is complete.
\end{proof}

 One has  to note that the above formula (\ref{th})
 holds under the request that $f$ must be at least  defined on $T_q$ and
$T_q^{1-\alpha}$.

\begin{definition}\label{qcd}
Let  $\alpha>0$. If $\alpha \notin \mathbb{N}$ , then the
$\alpha-$order Caputo left q-fractional and right q-fractional
derivatives of a function $f$ are, respectively,  defined by

\begin{equation} \label{qrd}
_{q}C_a^\alpha f(t)\triangleq~ _{q}I_a ^{(n-\alpha)}\nabla_q
^nf(t)=\frac{1}{\Gamma(n-\alpha)} \int_a^t(t-qs)_q^{n-\alpha-1}
\nabla_q^nf(s)\nabla_q s
\end{equation}
and

\begin{equation} \label{qld}
_{b}C_q^\alpha f(t)\triangleq ~_{b}I_q ^{(n-\alpha)}\nabla_b
^nf(t)=\frac{q^{(-1/2)\alpha(\alpha-1)}}{\Gamma_q(n-\alpha)}
\int_t^b(s-t)_q^{n-\alpha-1}~_{b}\nabla_q^nf(sq^{1-\alpha })\nabla_q
s
\end{equation}
where $n=[\alpha]+1$.

If $\alpha \in \mathbb{N}$, then $_{q}C_a^\alpha f(t)\triangleq
\nabla_q^n f(t)$ and $_{b}C_q^\alpha f(t)\triangleq~ _{b}\nabla_q^n=
(-1)^n \nabla_q^n$
\end{definition}
Also, it is clear that $_{q}C_a^\alpha $ maps functions defined on
$T_q$ to functions defined on $T_q$, and that $_{b}C_q^\alpha $ maps
functions defined on $T_q^{1-\alpha}$ to functions defined on $T_q$

 If, in Lemma \ref{m} and Lemma \ref{rm} we replace
$\alpha$ by $1-\alpha$ . Then, we can relate the left
and right Riemann and Caputo q-fractional derivatives. Namely, we
state

\begin{theorem} \label{qrelate}
For any $0<\alpha<1$, we have
\begin{equation}\label{relate1}
_{q}C_a^\alpha f(t)=~_{q}\nabla_a^\alpha f(t)-\frac{
(t-a)_q^{-\alpha}}{\Gamma_q(1-\alpha)} f(a)
\end{equation}
and

\begin{equation}\label{qrelate2}
 _{q^{-1}b}C_q^\alpha f(t)=~_{b}\nabla_q^\alpha
f(t)-\frac{r(1-\alpha)}{\Gamma_q(1-\alpha)}
 (b-qt)_q^{-\alpha}f(q^\alpha q^{-1}b)
\end{equation}

\end{theorem}

\section {A Caputo q-fractional Initial Value Problem and q-Mittag-Leffler
Function}\label{s:4}

The following identity which  is useful to transform Caputo q-fractional
difference equations  into q-fractional integrals, will be  our key in
this section.

\begin{proposition} \label{qtrans}
Assume $\alpha>0$ and $f$ is defined in suitable domains. Then
\begin{equation}\label{qtrans1}
_{q}I_a^\alpha~ _{q}C_a^\alpha
f(t)=f(t)-\sum_{k=0}^{n-1}\frac{(t-a)_q^{k}}{\Gamma_q(k+1)}\nabla_q^kf(a)
\end{equation}
and if $0<\alpha\leq1$ then
\begin{equation}\label{qtrans3}
_{q}I_a^\alpha~ _{q}C_a^\alpha f(t)= f(t)-f(a)
 \end{equation}

\end{proposition}
The proof followed by definition of Caputo q-fractional derivatives,
(\ref{t11}),  Lemma \ref{m} and Theorem \ref{qinter}.

The following identity \cite{Pred} is essential to solve linear
q-fractional equations
\begin{equation}\label{qpower}
_{q}I_a^\alpha (x-a)_q^\mu
=\frac{\Gamma_q(\mu+1)}{\Gamma_q(\alpha+\mu+1)}(x-a)_q^{\mu+\alpha}~~(0<a<x<b)
\end{equation}
where $\alpha \in \mathbb{R}^+$ and $\mu \in (-1,\infty)$.

\begin{example} \label{qlinear}
Let $0<\alpha\leq 1$ and consider the left Caputo q-fractional difference
equation
\begin{equation} \label{lfractional}
_{q}C^ \alpha _a y(t)= \lambda y(t)+f(t),~~y(a)=a_0,~t\in T_q.
\end{equation}
 if we apply $_{q}I_a^\alpha$ on the equation (\ref{lfractional}) then by
the help of (\ref{qtrans3}) we see that
$$y(t)= a_0+ \lambda~ _{q}I_a^\alpha y(t)+~_{q}I_a^\alpha f(t).$$
To obtain an explicit clear solution, we apply the method of successive
approximation. Set $y_0(t)=a_0$ and $$y_m(t)=a_0+\lambda~ _{q}I_a^\alpha
y_{m-1}(t)+ _{q}I_a^\alpha f(t), m=1,2,3,.... $$
For $m=1$, we have by the power formula (\ref{qpower})
 $$y_1(t)=a_0[1+\frac{\lambda ( t-a)_q^{(\alpha)}}{\Gamma_q(\alpha+1)}]+~
_{q}I_a^\alpha f(t).$$
For $m=2$, we also see that
$$y_2(t)= a_0+ \lambda a_0~_{q}I_a^\alpha[1+
\frac{(t-a)_q^\alpha}{\Gamma_q(\alpha+1)}]+~_{q}I_a^{\alpha} f(t)
+\lambda~ _{q}I_a^{2\alpha} f(t)$$
$$=a_0 [1+\frac{\lambda
(t-a)_q^\alpha}{\Gamma_q(\alpha+1)}+\frac{\lambda^2
(t-a)_q^{2\alpha}}{\Gamma_q(2\alpha+1)}]+~_{q}I_a^{\alpha} f(t) +\lambda~
_{q}I_a^{2\alpha} f(t)$$
If we proceed inductively and let $m\rightarrow\infty$ we obtain the solution
$$y(t)=a_0[1+\sum_{k=1}^\infty\frac{\lambda^k (t-a)_q^{k\alpha}}
{\Gamma_q(k\alpha+1)}]+\int_a^t
[\sum_{k=1}^{\infty}\frac{\lambda^{k-1}}{\Gamma_q(\alpha
k)}(t-qs)_q^{\alpha k-1}]f(s)\nabla_qs$$
$$ = a_0[1+\sum_{k=1}^\infty\frac{\lambda^k (t-a)_q^{k\alpha}}
{\Gamma_q(k\alpha+1)}]+\int_a^t
[\sum_{k=0}^{\infty}\frac{\lambda^{k}}{\Gamma_q(\alpha
k+\alpha)}(t-qs)_q^{\alpha k+(\alpha-1)}]f(s)\nabla_qs$$
$$= a_0[1+\sum_{k=1}^\infty\frac{\lambda^k (t-a)_q^{k\alpha}}
{\Gamma_q(k\alpha+1)}]+\int_a^t
(t-qs)_q^{(\alpha-1)}[\sum_{k=0}^{\infty}\frac{\lambda^{k}}{\Gamma_q(\alpha
k+\alpha)}(t-q^\alpha s)_q^{(\alpha k)}]f(s)\nabla_qs$$
If we set $\alpha=1$, $\lambda=1$, $a=0$ and $f(t)=0$ we come to a
q-exponential formula  $e_q(t)= \sum_{k=0}^\infty
\frac{t^k}{\Gamma_q(k+1)}$ on the time scale $T_q$, where
$\Gamma_q(k+1)=[k]_q!=[1]_q[2]_q...[k]_q$ with $[r]_q=\frac{1-q^r}{1-q}$.
It is known that $e_q(t)=E_q((1-q)t)$, where $E_q(t)$ is a special case of
the basic  hypergeormetric series, given by
$$E_q(t)=~_{1}\phi_0(0;q,t)=\Pi_{n=0}^\infty
(1-q^nt)^{-1}=\sum_{n=0}^\infty \frac{t^n}{(q)_n},$$ where
$(q)_n=(1-q)(1-q^2)...(1-q^n)$ is the q-Pochhammer symbol.
\end{example}
If we compare with the classical case, then the above example suggests the
following q-analogue of Mittag-Leffler function

\begin{definition}\label{Mitt}
For $z ,z_0\in \mathbf{C}$ and $\mathfrak{R}(\alpha)> 0$, the
q-Mittag-Leffler function is defined by
\begin{equation} \label{Mit}
_{q}E_{\alpha,\beta}(\lambda,z-z_0)=\sum_{k=0}^\infty \lambda^k
\frac{(z-z_0)_q^{\alpha k}}{\Gamma_q(\alpha k+\beta)}.
\end{equation}
When $\beta=1$ we simply use
$~_{q}E_{\alpha}(\lambda,z-z_0):=~_{q}E_{\alpha,1}(\lambda,z-z_0)$.
\end{definition}

According to Definition \ref{Mitt} above, the solution of the
q-Caputo-fractional equation in Example \ref{qlinear} is expressed by
$$y(t)=a_0 ~_{q}E_\alpha (\lambda,t-a)+ \int_a^t
(t-qs)_q^{\alpha-1}~_{q}E_{\alpha,\alpha} (\lambda,t-q^\alpha s)
f(s)\nabla_q s.$$

\begin{remark} \label{last}
1) Note that the above proposed definition of the q-analogue of
Mittag-Leffler function agrees with time scale definition of exponential
functions. As it depends on the three parameters other than $\alpha$ and
$\beta$.

2) The power term of the q-Mittag-Leffler function contains $\alpha$ (the
term $(z-z_0)_q^{\alpha k}$). We include this $\alpha$ in order to express
 the solution of q-Caputo initial value problem explicitly by means of the
q-Mittag-Leffler function. This is due to that in general it is not true
for the q-factorial function to satisfy the power formula
$(z-z_0)_q^{\alpha k}= [(z-z_0)_q^{\alpha }]^k$. But for example the
latter power formula is true when $z_0=0$. Therefore, for the case
$z_0=0$, we may drop $\alpha$ from the power so that the q-Mittag-Leffler
function will tend to the classical one when $q\rightarrow 1$.

3) Once Problem 1 raised in section 2 is solved an analogue result to
Proposition \ref{qtrans} can be obtained for right Caputo q-fractional
derivatives.
\end{remark}

\end{document}